\documentclass[a4paper,11pt]{amsart}
\usepackage{dsfont,hyperref,fancyhdr,mathrsfs,amsmath,amscd,amsthm,amsfonts,latexsym,amssymb,stmaryrd}
\usepackage[all]{xy}

\DeclareMathOperator{\grad}{grad}

\DeclareMathOperator{\dif}{d}

\renewcommand{\1}{\mathds{1}}

\newcommand{\V}{\mathscr{V}}

\def \a{\alpha}

\def \g{\gamma}

\def \o{\omega}
\def \O{\Omega}
\def \phi{\varphi}
\def \Phi{\varPhi}

\def \r{\rho}

\def \C{\mathbb{C}\,}

\def\widecheckg{g^{\hspace*{-2.5pt}\vbox to 5pt{\hbox to
0pt{\LARGE$\check{}$}}}\hspace*{2pt}}

\def\widecheckl{\lambda^{\hspace*{-3.5pt}\vbox to 8pt{\hbox to
0pt{\LARGE$\check{}$}}}\hspace*{2pt}}

\begin{document}

\title{On $G_2$-manifolds and geometry in dimensions $6$ and $8$}
\author{Radu Pantilie}  
%\thanks{} 
\address{R.~Pantilie, Institutul de Matematic\u a ``Simion~Stoilow'' al Academiei Rom\^ane,
C.P. 1-764, 014700, Bucure\c sti, Rom\^ania} 
\email{\href{mailto:Radu.Pantilie@imar.ro}{Radu.Pantilie@imar.ro}} 
\subjclass[2020]{53C29, 20G41} 
\keywords{$G_2$-manifolds, ${\rm Spin}(7)$-manifolds}

\newtheorem{thm}{Theorem}[section]
\newtheorem{lem}[thm]{Lemma}
\newtheorem{cor}[thm]{Corollary}
\newtheorem{prop}[thm]{Proposition}

\theoremstyle{definition}

\newtheorem{defn}[thm]{Definition}
\newtheorem{rem}[thm]{Remark}
\newtheorem{exm}[thm]{Example}

\numberwithin{equation}{section}

\thispagestyle{empty}

\begin{abstract}
We study the geometry induced on the local orbit spaces of Killing vector fields on (Riemannian) $G$-manifolds, with an emphasis on the cases 
$G={\rm Spin}(7)$ and $G=G_2$\,. Along the way, we classify the harmonic morphisms with one-dimensional fibres from $G_2$-manifolds to Einstein manifolds. 
\end{abstract} 

\maketitle 

\section*{Introduction}  

\indent 
It seems that, yet, there are no simple constructions of Riemannian manifolds with exceptional holonomy 
(see \cite{Bry-1987} and \cite{Sal-holo_book} for basic facts on 
these manifolds, and, also, Section \ref{section:rel_6}\,, below, for fairly explicit descriptions of the exceptional holonomy Lie algebras representations). 
To us, the archetypical (simple) construction of a `special' Riemannian metric is the classical Gibbons-Hawking construction 
(see \cite{Pan-hmhKPW}\,, and the refereneces therein) that characterises (for example), locally, the Ricci-flat anti-self-dual 
(equivalently, four-dimensional hyper-K\"ahler) manifolds endowed with a nowhere zero Killing vector field (see Remark \ref{rem:Killing_inf_auto}\,).\\ 
\indent 
We are, thus, led to the study of the geometry induced on the local orbit spaces of Killing vector fields on $G_2$-manifolds and ${\rm Spin}(7)$-manifolds 
(compare \cite{ApoSal}\,, \cite{Fow-quotient_Spin(7)}\,, \cite{Fow}\,). 
However, it is useful to start from a more general setting (see Section \ref{section:useful_setting}\,, below) 
whose easiest relevant case is provided by the, already mentioned, Ricci-flat anti-self-dual manifolds.\\ 
\indent 
In Section \ref{section:rel_6}\,, we present some results specific to dimension $6$ such as (Proposition \ref{prop:starting_rel_6}\,) 
the fact that, with respect to the orthogonal reductive decomposition $\mathfrak{g}_2=\mathfrak{su}(3)\oplus\mathfrak{m}$\,, the adjoint representation of 
$\mathfrak{su}(3)$ on $\mathfrak{m}$ is equivalent to its canonical representation on $\C^{\!3}$ (this, presumably known, fact led us to the useful setting  
of Section \ref{section:useful_setting}\,). The results presented in that section 
are, firstly, applied to characterise the harmonic morphisms from $G_2$-manifolds to Einstein manifolds of dimension $6$ (Corollary \ref{cor:hm_G2}\,).\\ 
\indent 
The results of Sections \ref{section:useful_setting} and \ref{section:rel_6} are applied, 
in Sections \ref{section:Killing_on_G2-manifolds} and \ref{section:Killing_on_Spin(7)-manifolds}\,, 
to obtain a better understanding of Killing vector fields on Riemannian manifolds with exceptional holonomy.\\ 
\indent 
It is fairly well known that the ($2$-)harmonic morphisms are in the background of the Gibbons-Hawking construction. 
Here (Remarks \ref{rem:G_2_with_B=0} and \ref{rem:for_Spin(7)}(1)\,) we have $4$-harmonic morphisms (see \cite{BaiWoo2} 
for more on this notion).\\  
\indent 
In Section \ref{section:sK}\,, some results on Killing vector fields on special K\"ahler manifolds are given.\\ 
\indent 
Finally, in Appendix \ref{section:appendix} we present what we believe to be the source of octonions.\\ 
\indent 
Other work, similarly motivated, can be found in \cite{ApoSal}\,, \cite{Cab}\,, \cite{Don-RemG2}\,, \cite{Fow-quotient_Spin(7)} and \cite{Fow}\,; 
see, also, \cite{Fow-to_appear} and \cite{MadSwa}\,.

\section{A useful setting} \label{section:useful_setting} 

\indent 
In this paper, for simplicity, unless otherwise specified, we work in the complex-analytic category.\\ 
\indent 
We are interested in the following setting.  

\begin{rem} \label{rem:starting_rel_6} 
Let $\mathfrak{g}$ be a Lie algebra endowed with a faithful orthogonal representation 
on an Euclidean space $U$ and a reductive de\-com\-po\-si\-tion $\mathfrak{g}=\mathfrak{h}\oplus\mathfrak{m}$ such that $\mathfrak{m}\subseteq U$, 
and the induced representation of $\mathfrak{h}$ on $U$ is the restriction of the given representation of $\mathfrak{g}$ on $U$.\\ 
\indent 
Let $G$ and $H$ be the simply-connected Lie groups with Lie algebras $\mathfrak{g}$ and $\mathfrak{h}$. 
Then if  $M$ is a manifold endowed with a reduction of its frame bundle to $H$ through its representation on $\mathfrak{m}$ we, also, have 
a vector bundle $E$ over $M$, with structural group $G$ and typical fibre $U$, endowed with a morphism of vector bundles 
$\r:E\to TM$ such that the following two facts hold:\\ 
\indent 
\quad$\bullet$ Any $\r$-connection \cite{Pan-qgfs} on $E$ corresponds to a connection on $E$ and a section of $({\rm ker}\r)^*\otimes{\rm End}\,E$  
(given by the orthogonal decomposition $U=\mathfrak{m}\oplus\mathfrak{m}^{\perp}$).\\  
\indent 
\quad$\bullet$ Any connection on $E$ (compatible with its structural group) corresponds to a connection on $M$ and a $(1,1)$-tensor field on $M$. 
\end{rem} 

\indent 
A significant particular case of Remark \ref{rem:starting_rel_6} is the following.  
Let $\mathfrak{g}\subseteq\mathfrak{so}(n+1)$ be a Lie algebra endowed with a faithful orthogonal 
representation of dimension $n+1$\,, and let $\mathfrak{h}=\mathfrak{g}\cap\mathfrak{so}(n)$\,, where $\mathfrak{so}(n)\subseteq\mathfrak{so}(n+1)$ 
is the Lie subalgebra preserving some nondegenerate hyperspace, where $n\in\mathbb{N}\setminus\{0\}$\,. 
Assuming $\mathfrak{g}$ nondegenerate, with respect to 
the Killing form of $\mathfrak{so}(n+1)$\,, and denoting by $\mathfrak{m}$ the orthogonal 
complement of $\mathfrak{h}$ in $\mathfrak{g}$\,, we obtain a reductive decomposition $\mathfrak{g}=\mathfrak{h}\oplus\mathfrak{m}$\,.\\ 
\indent 
We, further, assume $\dim\mathfrak{m}=n$ and the representation of $\mathfrak{h}$ on $\mathfrak{m}$\,, 
induced by the adjoint representation of $\mathfrak{g}$\,, be (up to some nonzero factor) the given $\mathfrak{h}\subseteq\mathfrak{so}(n)$\,. 
Therefore there exists a linear embedding $\mathfrak{m}\to\mathfrak{g}\subseteq\mathfrak{so}(n+1)$\,, given by 
\begin{equation*} 
x\mapsto 
\begin{pmatrix} 
h(x) & x \\ 
-x^T & 0 
\end{pmatrix} 
\;,
\end{equation*} 
for any $x\in\mathfrak{m}$\,, where $h:\mathfrak{m}\to\mathfrak{so}(n)$ is a linear map such that $[A,h(x)]=h(Ax)$\,, 
for any $A\in\mathfrak{h}$ and $x\in\mathfrak{m}$\,.\\ 
\indent 
Therefore if $\widetilde{A}=(A,x)\in\mathfrak{g}=\mathfrak{h}\oplus\mathfrak{m}$ then, as an element of $\mathfrak{so}(n+1)$\,, 
\begin{equation*} \label{e:Gamma} 
\widetilde{A}= 
\begin{pmatrix} 
A+h(x) & x \\ 
-x^T & 0 
\end{pmatrix} 
\;.
\end{equation*} 

\indent 
Now, we use Remark \ref{rem:starting_rel_6}\,, with $M$  a manifold endowed with an almost $H$-structure and $E=T(L\setminus0)/(\C\setminus\{0\})$\,, 
where $L$ is a line bundle over $M$ endowed with a connection. In particular, $E=TM\oplus\C$, and we denote by $\1$ the section of $E$ 
corresponding to $(0,1)$\,, and by $u$ the nowhere zero function on $M$ such that $u\1$ is given by the infinitesimal generator of the 
action of $\C\setminus\{0\}$ on $L$\,.\\ 
\indent 
It follows that any $\r$-connection $\widetilde{\nabla}$ on $E$, where $\r:E\to TM$ is the projection, corresponds to a $\r$-connection $\nabla$ on $TM$, 
and a section $\g$ of ${\rm Hom}(E,TM)$\,. These are related through the involved section $h$ of ${\rm Hom}(TM,{\rm End}(TM))$\,, 
as follows. 
\begin{equation} \label{e:nabla_tilde} 
\begin{split} 
\widetilde{\nabla}_XY&=\nabla_XY+h(\g(X))Y-<\g(X),Y>\1\;,\\ 
\widetilde{\nabla}_X\1&=\g(X)\;,\\ 
\widetilde{\nabla}_{\1}X&=\nabla_{\1}X+h(\g(\1))X-<\g(\1),X>\1\;,\\ 
\widetilde{\nabla}_{\1}\1&=\g(\1)\;, 
\end{split} 
\end{equation} 
for any (local) vector fields $X$ and $Y$ on $M$, where $<\cdot,\cdot>$ denotes the Riemannian metric on $M$ 
(and where, for simplicity, we did not mention the compatibilities with the structural groups; for example, $h$ is, in fact,  
a $1$-form on $M$ with values in the adjoint bundle of the underlying Riemannian structure on $M$, which is 
equivariant with respect to the actions of the adjoint bundle of the almost $H$-structure on $M$, where $H$ is the simply-connected Lie group 
with Lie algebra $\mathfrak{h}$\,).  

\begin{prop} \label{prop:torsion_tilde} 
The torsion $\widetilde{T}$ of $\widetilde{\nabla}$ is given by 
\begin{equation*} 
\begin{split} 
\widetilde{T}(X,Y)&=T(X,Y)+h(\g(X))Y-h(\g(Y))X+\bigl(u\O(X,Y)-<\g(X),Y>+<\g(Y),X>\bigr)\1\;,\\ 
\widetilde{T}(X,\1)&=\g(X)-\nabla_{\1}X-h(\g(\1))X+\bigl(u^{-1}X(u)+<\g(\1),X>\bigr)\1\;,\\ 
\end{split} 
\end{equation*} 
for any $X,Y\in TM$, where $T$ is the torsion of the connection given by $\nabla$ (and the connection on $L$), 
and $\O$ is the curvature form of the connection on $L$\,. 
\end{prop} 
\begin{proof} 
This follows quickly from \eqref{e:nabla_tilde}\,. 
\end{proof} 

\indent 
Let $G\subseteq O(k)$ be a Lie subgroup and $M$ a Riemannian manifold, $\dim M=k$\,, $k\in\mathbb{N}\setminus\{0\}$\,. 
We say that $M$ is a $G$-manifold if its orthonormal frame bundle admits a reduction $(P,M,G)$ 
such that the holonomy group of the Levi-Civita connection of $M$, with respect to some $u\in P$, is contained by $G$. 

\begin{rem} \label{rem:Killing_inf_auto} 
We consider only nowhere isotropic Killing vector fields $V$ on $G$-manifolds $M$ which are infinitesimal automorphisms of the given $G$-structure; 
that is, their flow preserves the reduction $(P,M,G)$\,. We, further, assume that the frames from $P$ whose first vector is 
given by $V$ (locally) normalized form a principal bundle preserved by the local flow of $V$.  
\end{rem} 

\begin{thm} \label{thm:Killing_on_G-manifolds} 
Let $M$ be a $G$-manifold, with $G$ connected and acting tranzitively on the unit (complex-)sphere. 
Let $H\subseteq G$ be the connected Lie subgroup whose Lie algebra is formed of those matrices of the Lie algebra of $G$ whose 
first rows and columns are zero.\\ 
\indent 
Let $M$ be endowed with a Killing vector field such that $H$ is the structural group of the bundle of frames 
whose first vector is given by $V$ normalized.\\  
\indent 
Then, locally, $M$ is of the form $N\times\C$ with the metric 
$<\cdot,\cdot>+u^2(\dif\!t+A)^2$, where $\bigl(N,<\cdot,\cdot>\bigr)$ is an almost $H$-manifold endowed with a compatible connection $\nabla$, 
a section $\mathcal{B}$ of its adjoint bundle, a nowhere zero function $u$, and a $1$-form $A$ such that, on denoting by $\g$ the $(1,1)$-tensor field 
on $N$ given by $\g(X)=\mathcal{B}X-u^{-1}h(\grad u)X$, for any $X\in TN$, the following relations hold
\begin{equation} \label{e:Killing_on_G-manifolds}
\begin{split} 
T(X,Y)&=-h(\g(X))Y+h(\g(Y))X\;,\\ 
\dif\!A(X,Y)&=2u^{-1}<\g(X),Y>\;,
\end{split} 
\end{equation}
for any $X,Y\in TN$, where $T$ is the torsion of $\nabla$; in particular, $\nabla+h\circ\g$ is the Levi-Civita connection of 
$\bigl(N,<\cdot,\cdot>\bigr)$. 
\end{thm} 
\begin{proof}
This follows quickly from Propositions \ref{prop:torsion_tilde}\,, where the $\r$-connection (which we denote by the same $\nabla$)   
is given by the connection $\nabla$ and $\mathcal{B}$.  
\end{proof}

\section{Basic facts in the geometry in dimension $6$} \label{section:rel_6} 

\indent 
The starting point of this section is the the following, most likely, known fact, where $\mathfrak{g}_2$ is the simple Lie algebra of dimension 14 
(see \cite{Pan-eholon} for an explicit description and, also, the proof of Proposition \ref{prop:starting_rel_6}\,, below). 

\begin{prop} \label{prop:starting_rel_6} 
Let $\mathfrak{g}_2\subseteq\mathfrak{so}(7)$ be given by the fundamental representation of dimension $7$ of $\mathfrak{g}_2$\,. 
Then under the equality of Lie algebras $\mathfrak{sl}(3)=\mathfrak{g}_2\,\cap\,\mathfrak{so}(6)$\,, the (adjoint) re\-pre\-sen\-ta\-tion of 
$\mathfrak{sl}(3)$ on its orthogonal complement in $\mathfrak{g}_2$ can be identified with the representation induced by the canonical representation of 
$\mathfrak{so}(6)$ (and with the direct sum of its $3$-dimensional fundamental representations). 
\end{prop} 
\begin{proof} 
We have to show that we are in the setting of Section \ref{section:useful_setting} with $\mathfrak{g}=\mathfrak{g}_2\subseteq\mathfrak{so}(7)$\,, 
and $\mathfrak{sl}(3)=\mathfrak{g}_2\,\cap\,\mathfrak{so}(6)$\,.\\ 
\indent 
We choose to deduce this by giving an explicit description, in matrix form, of the Lie algebras embeddings 
$\mathfrak{sl}(3)\subseteq\mathfrak{g}_2\subseteq\mathfrak{so}(7)$\,. For this, if $x\in\C^{\!3}$, we denote by $[x]$ the element of 
$\mathfrak{so}(3)$ such that $[x]y=x\times y$\,, for any $y\in\C^{\!3}$, where $\times$ denotes the quaternionic cross product 
(the well known explicit description of the isomorphism of Lie algebras between $(\C^{\!3},\times)$ and $\mathfrak{so}(3)$\,). 
Also, we use the characterisation of $\mathfrak{sl}(3)\subseteq\mathfrak{so}(6)$ as formed of all the matrices 
$\begin{pmatrix}
[x] &-y \\ 
y & [x]
\end{pmatrix} 
\;,$  
with $x\in\C^{\!3}$ and $y$ a $3\times3$ trace-free symmetric matrix.\\ 
\indent 
Now, we use the obvious embedding of $\mathfrak{so}(6)$ into $\mathfrak{so}(7)$ whose image has the fourths' rows and columns equal to zero.  
Also, we embed $\C^{\!6}$ into $\mathfrak{so}(7)$ as follows: to any $(a,b)\in\C^{\!6}=\C^{\!3}\oplus\C^{\!3}$ 
we associate $\{a\}^1+\{b\}^2$, where 
$$\{a\}^1= 
\begin{pmatrix}
0 & 2a & [a] \\ 
-2a^T & 0 & 0 \\ 
[a] & 0 & 0 
\end{pmatrix}
\;,\hspace*{5mm} 
\{b\}^2= 
\begin{pmatrix}
[b] & 0 & 0 \\ 
0 & 0 & -2b^T \\ 
0 & 2b & -[b] 
\end{pmatrix}
\;.$$ 
\indent 
From \cite{Pan-eholon} we deduce that $\mathfrak{g}_2$ is, inside $\mathfrak{so}(7)$\,, the direct sum of $\mathfrak{sl}(3)$ and $\C^{\!6}$, 
and the proof follows. 
\end{proof} 

\begin{rem} 
1) There is another way to look at Proposition \ref{prop:starting_rel_6}\,.
We have $\mathfrak{so}(6)+\mathfrak{g}_2=\mathfrak{so}(7)$\,, and $\mathfrak{so}(6)\cap\mathfrak{g}_2=\mathfrak{sl}(3)$\,. 
Consequently, the re\-pre\-sen\-ta\-tion of $\mathfrak{sl}(3)$ on its orthogonal complement in $\mathfrak{g}_2$ is given, by restriction, 
by the represention of $\mathfrak{so}(6)$ on its orthogonal complement in $\mathfrak{so}(7)$\,.\\ 
\indent 
2)  Similarly, let $\mathfrak{so}(7)_0$ and $\mathfrak{so}(7)_1$ 
be the images of the embeddings of $\mathfrak{so}(7)$ into $\mathfrak{so}(8)$ given by the canonical representation and the 
fundamental representation of dimension $8$ of $\mathfrak{so}(7)$\,, respectively.\\ 
\indent 
Then $\mathfrak{so}(7)_0+\mathfrak{so}(7)_1=\mathfrak{so}(8)$\,, and $\mathfrak{so}(7)_0\cap\mathfrak{so}(7)_1=\mathfrak{g}_2$\,. 
Moreover, both embeddings of $\mathfrak{g}_2$ into $\mathfrak{so}(7)_j$\,, $j=0,1$\,, are given by the fundamental representation 
of dimension $7$ of $\mathfrak{g}_2$\,. Consequently, the representation of $\mathfrak{so}(7)_1$ on $\mathfrak{so}(8)$ 
decomposes as the direct sum of its adjoint and canonical representations. Furthermore, it, also, follows that the reductive decompositions  
induced by the embeddings of $\mathfrak{so}(7)_j$ into $\mathfrak{so}(8)$\,, $j=0,1$\,, are the same (symmetric decomposition).\\ 
\indent 
3) Also, the embedding of $\mathfrak{so}(7)$ into $\mathfrak{so}(8)$\,, corresponding to the $8$-dimensional fundamental representation of the 
former, admits an explicit description, similar to the one given, in Proposition \ref{prop:starting_rel_6}\,, for 
$\mathfrak{g}_2\subseteq\mathfrak{so}(7)$\,. This follows, for example, from the latter 
and the formula for $h$ given in Section \ref{section:Killing_on_Spin(7)-manifolds}\,, below.  
\end{rem} 

\indent 
Let $Q^5\subseteq{\rm Gr}_3^0(7)$ be the embedding corresponding to the Lie algebras embedding $\mathfrak{g}_2\subseteq\mathfrak{so}(7)$\,, 
and let $V\subseteq U_{1,0}$ be the (nondegenerate) subspace of dimension $6$ giving the equality 
$\mathfrak{sl}(3)=\mathfrak{g}_2\,\cap\,\mathfrak{so}(6)$\,, where $U_{1,0}$ is the fundamental representaton space of dimension $7$ of $\mathfrak{g}_2$\,.   
Recall that ${\rm Gr}_3^0(V)=PU\sqcup PU^*$, where $U$ is a vector space of dimension $4$ 
such that $V=\Lambda^2U$. Also, note that, the Euclidean structure on $V$ (giving $\mathfrak{so}(6)$\,) 
is induced by an (oriented) Euclidean structure on $U$.  
Although this induces an isomorphism from $U$ onto $U^*$ we will keep the distinction 
between them to emphasize that $PU$ parametrizes `positive' (isotropic) subspaces whilst $PU^*$ `negative' subspaces.  

\begin{prop} \label{prop:rel_6_other_facts}  
{\rm (i)} The subspace of points of $Q^5\bigl(\subseteq{\rm Gr}_3^0(7)\bigr)$ which as $3$-dimensional isotropic spaces are 
contained by $V$ is of the form $PW\sqcup PW^*$, where $W\subseteq U$ has codimension $1$\,.\\ 
\indent 
{\rm (ii)} Let $Q^4={\rm Gr}_1^0(V)$\,. There exists a surjective map from $Q^4\setminus\bigl(PW\sqcup PW^*\bigr)$  
onto $\bigl\{p\cap V\,|\,p\in Q^5,\,\dim(p\cap V)=2\,,\,p\cap V\textrm{anti-self-dual}\bigr\}$ characterised by $\ell\mapsto p\cap V$ 
if and only if there exists a nondegenerate associative space $q\subseteq V\bigl(\subseteq U_{1,0}\bigr)$ such that $\ell=p\cap q$\,.  
\end{prop}   
\begin{proof} 
(i) Let $V_{\pm}=\Lambda_{\pm}^2U$ and, note that, $\mathfrak{sl}(3)\subseteq\mathfrak{gl}(3)$\,, 
where $\mathfrak{gl}(3)\subseteq\mathfrak{so}(V)$ is characterised by its com\-pa\-ti\-bi\-li\-ty with the 
orthogonal complex structure mapping $(x_+,x_-)\in V=V_+\oplus V_-$ to $(-x_-,x_+)$\,, 
where we use the orientation preserving isometry between $V_+$ and $V_-$ given by the octonionic cross product on $U_{1,0}$\,. 
Let $W_{\pm}$ be the $(\pm{\rm i})$-eigenspaces 
of this orthogonal complex structure. Then $W_+$ and $W_-$ are points in $PU$ and $PU^*$, respectively, and $W\subseteq U$ is the annihilator 
of $W_-$ as a $1$-dimensional subspace of $U^*$.\\ 
\indent 
(ii) Any point of $Q^5$ is uniquely determined by an isotropic direction in $U_{1,0}$ and a nondegenerate associative space containing it. 
Let $\ell\in Q^4\setminus\bigl(PW\sqcup PW^*\bigr)$\,. We have to show that if $q_1$ and $q_2$ are nondegenerate associative spaces contained by $V$ 
such that $\ell=q_1\cap q_2$ then the points of $Q^5$ determined by $(\ell,q_1)$ and $(\ell,q_2)$ have the same intersection with $V$.  
This follows quickly from the following two facts:\\ 
\indent 
\quad(ii1) $\dim\bigl(q_1^{\perp}\cap q_2^{\perp}\bigr)=2$\,, where $q_j^{\perp}$ is the orthogonal complement of $q_j$ in $U_{1,0}$\,, $j=1,2$\,;\\ 
\indent 
\quad(ii2) $q_1^{\perp}\cap q_2^{\perp}$ is degenerate and contains a unique isotropic direction, as it, also, contains $V^{\perp}$. 
\end{proof}     

\begin{rem} \label{rem:nK_submfds} 
Let $N$ be a Riemannian manifold of dimension $6$ endowed with a reduction $P$ of its orthonormal frame bundle to ${\rm GL}(3)$\,, 
corresponding to an almost Hermitian structure $J$. Similarly to Remark \ref{rem:starting_rel_6}\,, the Levi-Civita connection $\nabla$ of $N$ 
corresponds to a connection on $P$ and a $(1,1)$-tensor field $A$ on $N$.\\ 
\indent 
Then $(N,J)$ is nearly-K\"ahler (that is $(\nabla_XJ)(X)=0$\,, for any $X\in TN$) if and only if, pointwisely, $A$ is in the space 
generated by ${\rm Id}_{TN}$ and $J$. 
\end{rem} 

\indent 
Unless otherwise stated, any nondegenerate hypersurface in an almost $G_2$-manifold will be considered endowed with the almost Hermitian structure 
given by Propositions \ref{prop:starting_rel_6} and \ref{prop:rel_6_other_facts}\,. 

\begin{cor} \label{cor:nK_submfds} 
A nondegenerate hypersurface of a $G_2$-manifold is nearly-K\"ahler/K\"ahler if and only if it is umbilical/geodesic. 
\end{cor} 
\begin{proof} 
This is a straightforward consequence of Proposition \ref{prop:starting_rel_6} and Remark \ref{rem:nK_submfds}\,. 
\end{proof} 

\begin{cor} \label{cor:hm_G2} 
Let $\phi$ be a \emph{real} harmonic morphism with $1$-dimensional fibres from a $G_2$-manifold to an Einstein manifold.\\ 
\indent 
Then the codomain of $\phi$ is nearly-K\"ahler and its fibres are geodesics orthogonal to an umbilical foliation by hypersurfaces. 
\end{cor}
\begin{proof} 
This follows from \cite{PanWoo-d} and Corollary \ref{cor:nK_submfds}\,. 
\end{proof}

\section{Killing vector fields on $G_2$-manifolds} \label{section:Killing_on_G2-manifolds} 

\indent 
Approaches different to the one of this section can be found in \cite{ApoSal} and \cite{Fow}\,.\\  
\indent 
To give an almost ${\rm SL}(3)$-structure, through the direct sum of its fundamental representations, on a six-dimensional Riemannian manifold $N$
is the same as to specify the bundle of nondegenerate `associative' spaces contained by $TN$. 
Here, associativity is induced by the octonionic cross product on the bundle $E=TN\oplus\C$, 
obtained through the setting of Section \ref{section:useful_setting} (and with $M$ replaced by $N$), 
by taking $\mathfrak{g}=\mathfrak{g}_2$ and $\mathfrak{h}=\mathfrak{sl}(3)$ (see, also, Proposition \ref{prop:starting_rel_6} 
whose explicit formulae for the embedding $\mathfrak{sl}(3)\subseteq\mathfrak{g}_2$ require a choice of a nondegenerate associative space). 

\begin{cor} \label{cor:Killing_on_G2-manifolds} 
Any $G_2$-manifold endowed with a Killing vector field is locally of the form $N\times\C$ with the metric 
$<\cdot,\cdot>+u^2(\dif\!t+A)^2$, where $\bigl(N,<\cdot,\cdot>\bigr)$ is an almost ${\rm SL}(3)$-manifold endowed with a compatible connection $\nabla$, 
a section $\mathcal{B}$ of its adjoint bundle, a nowhere zero function $u$, and a $1$-form $A$ such that, on denoting by $\g$ the $(1,1)$-tensor field 
on $N$ given by $\g(X)=\mathcal{B}X-u^{-1}h(\grad u)X$, for any $X\in TN$, the following conditions hold\\ 
\indent 
\quad{\rm (i)} $\nabla+h\circ\g$ is the Levi-Civita connection of $\bigl(N,<\cdot,\cdot>\bigr)$\,,\\ 
\indent 
\quad{\rm (ii)} $\dif\!A(X,Y)=2u^{-1}<\g(X),Y>$\,, for any $X,Y\in TN$. 
\end{cor} 
\begin{proof}
This follows quickly from Theorem \ref{thm:Killing_on_G-manifolds} and Proposition \ref{prop:starting_rel_6}\,. 
\end{proof} 

\begin{rem} \label{rem:Killing_on_G2-manifolds} 
In Corollary \ref{cor:Killing_on_G2-manifolds}\,, $\nabla$ is the Levi-Civita connection of $N$ if and only if 
the Killing vector field corresponds, locally, to isometries between $M$ and the Riemannian product of special K\"ahler manifolds 
and open subsets of the real line. Moreover, this is equivalent to \emph{(a)} the integrability of the orthogonal complement distribution 
to the Killing vector field, and, in the real setting (similarly to Corollary \ref{cor:hm_G2}\,), to \emph{(b)} the orbits of the Killing vector field be geodesics.\\ 
\indent 
A similar statement holds in the setting of Theorem \ref{thm:Killing_on_G-manifolds} if we assume $h:\mathfrak{m}\to\mathfrak{so}(n)$ injective. 
For example, by taking (1) $G={\rm Spin}(7)$, with its fundamental representation of dimension $8$\,, and $H=G_2$\,, 
(2) $G={\rm SL}(2)$ and $H$ trivial, with $G$ endowed with the representation on $\mathfrak{gl}(2)$\,, given by left multiplication. 
In the latter case, with the same notations as in Theorem \ref{thm:Killing_on_G-manifolds}\,, we have $h(X)Y=X\times Y$, for any $X,Y\in TN$, 
and $\g(X)=-u^{-1}h(\grad u)X$\,, for any $X\in TN$\,; consequently, the first relation of \eqref{e:Killing_on_G-manifolds} 
is equivalent to the fact that $u^2\!<\cdot,\cdot>$ is flat. Thus, in this case, Theorem \ref{thm:Killing_on_G-manifolds} 
reduces to the classical Gibbons-Hawking construction (see \cite{Pan-hmhKPW}\,, and the references therein).\\ 
\indent 
See Section \ref{section:Killing_on_Spin(7)-manifolds}\,, below, for the relevant $h$ in case (1)\,.   
\end{rem} 

\indent 
To make the conditions of Corollary \ref{cor:Killing_on_G2-manifolds} more explicit, we endow the tangent bundle of $N$ 
with an `associative (orthogonal) decomposition' $TN=T^+N\oplus T^-N$.\\ 
\indent 
Then we can write 
$\mathcal{B}=\begin{pmatrix}
[b] & -B \\ 
B & [b]  
\end{pmatrix}$\,, where $b$ is a section of $T^+N$ and $B$ is a trace-free self-adjoint section of ${\rm Hom}(T^+N,T^-N)$\,, 
where we use the orientation preserving isometry between $T^+N$ and $T^-N$ given by the octonionic cross product on $TN\oplus\C$.\\ 
\indent 
Also, we have  $h\begin{pmatrix}
X_+\\ 
X_-
\end{pmatrix}
=\frac12\begin{pmatrix}
[X_-] & [X_+] \\ 
[X_+] & -[X_-] 
\end{pmatrix}$\,, for any $X_{\pm}\in T^{\pm}N$.\\ 
\indent 
Consequently, 
\begin{equation} \label{e:gamma} 
\g\begin{pmatrix}
X_+ \\ 
X_-
\end{pmatrix}
=\begin{pmatrix}
b\times X_+ -BX_- -\frac12\,u^{-1}(\grad u)_-\times X_+ -\frac12\,u^{-1}(\grad u)_+\times X_-\\ 
BX_+ +b\times X_- -\frac12\,u^{-1}(\grad u)_+\times X_+ +\frac12\,u^{-1}(\grad u)_-\times X_-
\end{pmatrix}\;,
\end{equation} 
for any $X_{\pm}\in T^{\pm}N$. 

\indent 
Hence, condition (ii) of Corollary \ref{cor:Killing_on_G2-manifolds} becomes  
\begin{equation} \label{e:ii_G2}
\begin{split} 
\dif\!A(X_+,Y_+)&=2u^{-1}\det\bigl(b-\tfrac12\,u^{-1}(\grad u)_-,X_+,Y_+\bigr)\;,\\  
\dif\!A(X_-,Y_-)&=2u^{-1}\det\bigl(b+\tfrac12\,u^{-1}(\grad u)_-,X_-,Y_-\bigr)\;,\\ 
\dif\!A(X_+,Y_-)&=2u^{-1}\bigl(<BX_+,Y_->-\tfrac12\,u^{-1}\det((\grad u)_+,X_+,Y_-)\bigr)\;,\\ 
\dif\!A(X_-,Y_+)&=2u^{-1}\bigl(-<BX_-,Y_+>-\tfrac12\,u^{-1}\det((\grad u)_+,X_-,Y_+)\bigr)\;, 
\end{split} 
\end{equation}
for any $X_{\pm},Y_{\pm}\in T^{\pm}N$, where, obviously, the last two equations give the same condition. 

\begin{rem} \label{rem:useful_particular_case} 
1) If $b=\tfrac12\,u^{-1}(\grad u)_-$\,, $B=0$\,, and $(\grad u)_+=0$ then \eqref{e:gamma} becomes  
$$\g\begin{pmatrix}
X_+ \\ 
X_-
\end{pmatrix}
=\begin{pmatrix}
0\\ 
u^{-1}(\grad u)_-\times X_-
\end{pmatrix}\;,$$
for any $X_{\pm}\in T^{\pm}N$. Therefore, in this case, condition (ii) of Corollary \ref{cor:Killing_on_G2-manifolds} becomes 
\begin{equation} \label{e:useful_particular_case}
\begin{split} 
\dif\!A(X_-,Y_-)&=2u^{-2}\det((\grad u)_-,X_-,Y_-)\;,\\ 
\dif\!A(X_+,Y_{\pm})&=0\;,
\end{split} 
\end{equation}
for any $X_{\pm},Y_{\pm}\in T^{\pm}N$\,; equivalently, $\dif\!A=2u^{-2}*_-\dif\!u$\,, where 
$*_-$ is the Hodge \mbox{$*$-operator} on $T^-N$.\\ 
\indent 
Now, we change $<\cdot,\cdot>$ by the conformal factor $u^2$ along $T^-N$ and by keeping it unchanged elsewhere; denote by $<\cdot,\cdot>_1$ 
this new metric on $N$. Then \eqref{e:useful_particular_case} becomes 
\begin{equation} \label{e:monopole_for_G2}
\dif\!A=-*^1_-\dif(u^{-2})\;,
\end{equation}
where $*^1_-$ is the Hodge \mbox{$*$-operator} on $T^-N$, with respect to $<\cdot,\cdot>_1$\,.\\ 
\indent 
If $T^{\pm}N$ are integrable and $\nabla$ is compatible with the decomposition $TN=T^+N\oplus T^-N$ then 
(i) of Corollary \ref{cor:Killing_on_G2-manifolds} holds if and only if the Levi-Civita connection of $<\cdot,\cdot>_1$ 
is given by $$X\mapsto\nabla_X+u^{-1}X(u){\rm Id}_{TN}+\tfrac12\begin{pmatrix}
u^{-1}[(\grad u)_-\times X_-] & 0 \\ 
0 &  u^{-1}[(\grad u)_-\times X_-] 
\end{pmatrix}\;,$$ 
for any $X=X_++X_-\in TN=T^+N\oplus T^-N$. Consequently, $(N,<\cdot,\cdot>_1)$ is a ${\rm SL}(3)$-manifold.\\ 
\indent 
Assume, further, that, up to a suitable (local) frame on 
$T^+N\,(=T^-N)$ the local connection form of $\nabla$ is given by the $\mathfrak{so}(3)$-valued $1$-form $-\tfrac12\,u^{-1}[(\grad u)_-]$ 
(under the obvious diagonal embedding $\mathfrak{so}(3)\subseteq\mathfrak{sl}(3)$\,). 
Then $T^{\pm}N$ are (integrable and) totally geodesic with respect to $<\cdot,\cdot>$\,; consequently, 
the same holds for $<\cdot,\cdot>_1$\,. Thus, locally, $(N,<\cdot,\cdot>_1)$ is a product $N^+\times N^-$, 
where $N^{\pm}$ are leaves of $T^{\pm}N$. Then (i) of Corollary \ref{cor:Killing_on_G2-manifolds} holds if and only if $N^{\pm}$ are flat.   
Consequently, by applying Corollary \ref{cor:Killing_on_G2-manifolds} we obtain a $G_2$-manifold which, locally, is a product 
of a flat Euclidean space of dimension $3$\,, and a Ricci-flat anti-self-dual manifold obtained by applying the Gibbons-Hawking construction.\\ 
\indent 
2) If $b=0$\,, $B=0$\,, and $(\grad u)_+=0$ (or $(\grad u)_-=0$) then (i) of Corollary \ref{cor:Killing_on_G2-manifolds} is equivalent to 
$(N,u\!<\cdot,\cdot>)$ is K\"ahler.\\ 
\indent 
3) If $(\grad u)_+=0$ (or $(\grad u)_-=0$) then (i) of Corollary \ref{cor:Killing_on_G2-manifolds} is equivalent to the fact that the Levi-Civita connection of 
$(N,u\!<\cdot,\cdot>)$ is $\nabla^1+h\circ\mathcal{B}$, where $\nabla^1$ is a connection on $N$ compatible with the 
almost Hermitian structure of $(N,u\!<\cdot,\cdot>)$\,. 
\end{rem}

\begin{cor} \label{cor:for_G2} 
With the same hypotheses and notations as in Corollary \ref{cor:Killing_on_G2-manifolds}\,, 
suppose that the set where $\grad u$ is an eigenvector of the underlying almost complex structure has empty interior.\\ 
\indent 
Then the Levi-Civita connection of $(N,u\!<\cdot,\cdot>)$ is equal to $\nabla^1+h\circ\mathcal{B}$, 
where $\nabla^1$ is a connection on $(N,u\!<\cdot,\cdot>)$ compatibe with its almost Hermitian structure; 
in particular, if $\mathcal{B}=0$ then $(N,u\!<\cdot,\cdot>)$ is K\"ahler; moreover, in the latter case, 
$u^{1/2}\o$ and $u^{1/2}\widetilde{\o}$ are closed, for any forms $\o$ and $\widetilde{\o}$\,, 
of types $(3,0)$ and $(0,3)$\,, respectively, such that $\nabla\o=0=\nabla\widetilde{\o}$\,. 
\end{cor} 
\begin{proof} 
It is sufficient to prove that $(N,u\!<\cdot,\cdot>)$ is K\"ahler, at least, outside a closed set with empty interior.\\ 
\indent  
If $u$ is constant then the proof follows from Corollary \ref{cor:Killing_on_G2-manifolds}\,. 
Consequently, we may assume that $u$ has no critical points. Then, at least outside the set $S$ where $\grad u$ is isotropic, locally, 
we may endow $N$ with an associative decomposition $TN=T^+N\oplus T^-N$, 
such that $(\dif\!u)|_{T^-N}=0$\,. Furthermore, the same holds on the interior of $S$, and the proof follows from (2) and (3) of 
Remark \ref{rem:useful_particular_case}\,. 
\end{proof}

\begin{rem} \label{rem:G_2_with_B=0}  
Let $(N,k)$ be a K\"ahler manifold, $\dim N=6$\,, endowed with a nowhere isotropic form $\a$ of type $(3,0)\oplus(0,3)$\,. 
Suppose that $v^{-1/3}\a$ and $\iota_{v^{-1}\!\grad v}\a$ are closed, where $v^2=k(\a,\a)$\,.\\ 
\indent 
Then, the locally defined, $\bigl(\C\!\times N,v^{2/3}k+v^{-4/3}(\dif\!t+A)^2\bigr)$ is a $G_2$-manifold, where $A$ is a $1$-form on $N$ such that 
$\dif\!A=-\tfrac43\,\iota_{v^{-1}\!\grad v}\a$\,. Moreover, any $G_2$-manifold endowed with a Killing vector field 
is, locally, obtained this way if, with the same notations 
as in Corrolary \ref{cor:Killing_on_G2-manifolds}\,, $\mathcal{B}=0$ and 
the set where $\grad u$ is an eigenvector of the underlying complex structure has empty interior.\\ 
\indent 
This follows from Corollaries \ref{cor:Killing_on_G2-manifolds} and \ref{cor:for_G2}\,.\\ 
\indent 
Also, we have ${\rm div}(\grad v)+v^{-1}k(\dif\!v,\dif\!v)=0$\,; equivalently, the locally defined, $\log v$ is harmonic.\\  
\indent 
If $(N,k)$ is flat special K\"ahler then such $\a$ can be found, but only covariantly constants in the real setting.  
\end{rem}

\section{Killing vector fields on ${\rm Spin}(7)$-manifolds} \label{section:Killing_on_Spin(7)-manifolds} 

\indent 
With the same notations as in Section \ref{section:useful_setting}\,, here we consider the case $G={\rm Spin}(7)$ and $H=G_2$ (compare \cite{Fow-quotient_Spin(7)}\,). 
Then a result similar Corollary \ref{cor:Killing_on_G2-manifolds} follows from Theorem \ref{thm:Killing_on_G-manifolds}\,. 
Further, with $U$ the space of imaginary octonions, $h:U\to\mathfrak{so}(U)$ is given by $h(X)=-\tfrac13X\times(\cdot)$\,, for any $X\in U$, 
where $\times$ denotes the octonionic cross product.\\ 
\indent  
Similarly to Section \ref{section:Killing_on_G2-manifolds}\,, an orthogonal decomposition $U=U_+\oplus U_-$ with $U_+$ (nondegenerate and) associative, 
will be called \emph{associative}.   

\begin{lem} \label{lem:double_octon_cross_prod}
Let $U=U_+\oplus U_-$ be an associative decomposition, and let $u\in U_+$ and $X\in U$. 
Then $(u\times X)\times(\cdot)$ is given by $-3(u_aX_b-u_bX_a)_{a,b}$\,, 
up to an element of $\mathfrak{g}_2(\subseteq\mathfrak{so}(U))$\,, where $u_a$ and $X_a$ are the components 
of $u$ and $X$, respectively, with respect to an octonionic basis of $U$, adapted to the given associative decomposition. 
\end{lem} 
\begin{proof} 
This follows from a straightforward computation. 
\end{proof}  
 
\begin{cor} \label{cor:for_Spin(7)} 
Under the hypotheses and notations of Theorem \ref{thm:Killing_on_G-manifolds}\,, with $G={\rm Spin}(7)$ and $H=G_2$\,, 
assume $\mathcal{B}=0$\,.\\ 
\indent 
Then (locally) $(N,u^{2/3}\!<\cdot,\cdot>)$ is a $G_2$-manifold which, if $u$ is nonconstant, is the product of a special K\"ahler manifold, of dimension $6$\,, 
and an open subset of $\C$. 
\end{cor} 
\begin{proof} 
It is sufficient to prove that $(N,u^{2/3}\!<\cdot,\cdot>)$ is a $G_2$-manifold, at least, outside a closed set with empty interior.\\ 
\indent  
If $u$ is constant then the proof follows from Theorem \ref{thm:Killing_on_G-manifolds}\,. 
Consequently, we may assume that $u$ has no critical points. Then, at least outside the set $S$ where $\grad u$ is isotropic, locally, 
we may endow $N$ with an associative decomposition $TN=T^+N\oplus T^-N$, 
such that $(\dif\!u)|_{T^-N}=0$\,. Furthermore, the same holds on the interior of $S$, and from Theorem \ref{thm:Killing_on_G-manifolds} 
and Lemma \ref{lem:double_octon_cross_prod} it follows that $(N,u^{2/3}\!<\cdot,\cdot>)$ is a $G_2$-manifold, at least, outside the frontier of $S$.\\ 
\indent 
If $u$ is nonconstant then, as $\mathcal{B}=0$\,, we have on $(N,u^{2/3}\!<\cdot,\cdot>)$ a nontrivial covariantly constant vector field, thus, 
completing the proof.  
\end{proof}

\begin{rem} \label{rem:for_Spin(7)} 
1) In Corrollary \ref{cor:for_Spin(7)}\,, if $(N,u^{2/3}\!<\cdot,\cdot>)$ is given by a product of a hyper-K\"ahler manifold, of dimension $4$\,, 
and a domain of an Euclidean space, of dimension $3$\,, with $u$ constant along the submanifolds given by the former, then,  
up to a constant, $u^{-5/3}$ must be linear.\\ 
\indent  
Conversely, by starting with a linear function defined on a suitable domain of the space of imaginary quaternions then, by using  
Theorem \ref{e:Killing_on_G-manifolds} and Corollary \ref{cor:for_Spin(7)}\,, we can build a ${\rm Spin}(7)$-manifold.\\ 
\indent 
More generally, we can start from a product as in Corollary \ref{cor:for_Spin(7)}\,.\\   
\indent 
2) Without the assumption $\mathcal{B}=0$\,, the conclusion of Corollary \ref{cor:for_Spin(7)} is that the Levi-Civita connection of 
$(N,u^{2/3}\!<\cdot,\cdot>)$ is $\nabla^1+h\circ\mathcal{B}$\,, where $\nabla^1$ is a connection on $N$ compatible with the 
almost $G_2$-structure of $(N,u^{2/3}\!<\cdot,\cdot>)$\,. 
\end{rem}

\section{Killing vector fields on special K\"ahler manifolds} \label{section:sK} 

\indent 
With the same notations as in Section \ref{section:useful_setting}\,, here we consider the case $G={\rm SL}(n)$ 
and $H={\rm SL}(n-1)$\,, where the former is considered with the direct sum of its fundamental representations of dimensions $n\geq2$\,. 
Similarly to the previous sections, these are the eigenspaces of an orthogonal complex structure $J$ on an Euclidean space $\widetilde{U}$, 
$\dim\widetilde{U}=2n$\,. Furthermore, the induced representation space of ${\rm SL}(n-1)$ is a codimension $1$ subspace $U\subseteq\widetilde{U}$  
admitting an orthogonal decomposition $U=\C\!u_0\oplus U_+\oplus U_-$\,, where $u_0\in U$ is such that $<u_0,u_0>=1$\,, $Ju_0\in U^{\perp}$, and 
$J|_{U_{\pm}}:U_{\pm}\to U_{\mp}$ are isometries; in particular, $\dim U_{\pm}=n-1$\,. Note that, ${\rm SL}(n-1)$ acts trivially on $u_0$\,, 
and `canonically' on $U_+\oplus U_-$ (the latter also holds for ${\rm SL}(n)$\,, with respect to $\widetilde{U}_+=U^{\perp}\oplus U_+$ 
and $\widetilde{U}_-=\C\!u_0\oplus U_-$\,).\\ 
\indent  
Consequently, the resulting $h:U\to\mathfrak{so}(U)$ is characterised by the following:\\ 
\indent 
\quad(i) $h(u_0)|_{U_+\oplus U_-}=-\tfrac{1}{n-1}J|_{U_+\oplus U_-}$\,, $h(u_0)u_0=0$\,.\\ 
\indent 
\quad(ii) $h(u_+)u_0=Ju_+$\,, $h(u_+)v_+=0$\,, $h(u_+)u_-=-<Ju_+,u_->u_0$\,, for any $u_+,v_+\in U_+$ and $u_-\in U_-$\,.\\ 
\indent 
\quad(iii) $h(u_-)u_0=Ju_-$\,, $h(u_-)v_-=0$\,, $h(u_-)u_+=-<Ju_-,u_+>u_0$\,, for any $u_-,v_-\in U_-$ and $u_+\in U_+$\,. 

\begin{rem} 
1) The corresponding map $h:U\times U\to U$ is skew-symmetric if and only if $n=2$ is which case, as mentioned in 
Remark \ref{rem:Killing_on_G2-manifolds}\,, it is the quaternionic cross product.\\ 
\indent 
2) Conditions (ii) and (iii) can be written $h(u)u_0=Ju$ and $h(u)v=-<Ju,v>u_0$\,, for any $u,v\in u_0^{\perp}$. 
Furthermore, as $h(u)u_0\in u_0^{\perp}$, for any $u\in U$, we have that (iii) implies (ii)\,. 
\end{rem} 

\begin{cor} \label{cor:sK} 
If $n\geq3$ and $\mathcal{B}=0$\,, the distribution $u_0^{\perp}$ on $(N,<\cdot,\cdot>)$ is integrable, umbilical and its mean curvature is equal 
to $\tfrac{1}{n-1}\,u_0(u)u_0$\,. Consequently, locally, $u_0^{\perp}$ is geodesic with respect to $u^{2/n-1}\!<\cdot,\cdot>$.  
\end{cor} 
\begin{proof}
This follows from conditions (i)\,, (ii)\,, (iii)\,, above, and Theorem \ref{thm:Killing_on_G-manifolds}\,. 
\end{proof}

\begin{cor} \label{cor:sK_conf} 
Let $M$, $\dim M\geq4$\,, be a special K\"ahler manifold endowed with a Killing vector field $V$. 
Let $\V$ be the foliation on $M$ given by $V$ and $JV$.\\ 
\indent  
If $\mathcal{B}=0$ then $\V$ is a conformal foliation. 
\end{cor} 
\begin{proof} 
If $\dim M=4$ this is well known (and fairly obvious; see \cite{BaiWoo2}\,).\\ 
\indent 
If $\dim M\geq6$\,, by Corollary \ref{cor:sK}\,, locally, $\V$ is given by the composition of a Riemannian submersion folowed by a 
horizontally conformal submersion, with nondegenerate fibres. 
\end{proof}

\appendix

\section{Where do the octonions come from?} \label{section:appendix} 

\indent 
It is known (see \cite{Pan-qPt} and the references therein) that the quaternions arise from the Riemann sphere $Y$, endowed with the antipodal map. 
To show this, note that, the antipodal map lifts to $TY$ and therefore to $(TY\setminus0)/(\C\!\setminus\{0\})$\,, aswell. 
Then the fixed point set of the antipodal map on $H^0\bigl((TY\setminus0)/(\C\!\setminus\{0\})\bigr)$ 
is just the algebra of quaternions.\\ 
\indent 
For the octonions, we start from $Q^5$ and observe that the space $Q'$ of projective planes contained by $Q^5$ 
can be identified with $Q^6$ (see \cite{Pan-eholon}\,). Hence, $Q^5$ may be embedded into $Q'$, thus, providing, up to a conjugation,  
the necessary and sufficient ingredient to build the octonions. 
Indeed, we may define $\mathfrak{g}_2$ as the Lie subalgebra of $\mathfrak{so}(7)$ (represented as a Lie algebra of vector fields on $Q'$) 
that preserves $Q^5$. Then the (complex) octonionic cross product is, for example, the torsion of the 
reductive decomposition $\mathfrak{so}(7)=\mathfrak{g}_2\oplus(\mathfrak{g}_2)^{\perp}$.

\end{document}